\documentclass[12pt]{article}
\usepackage{array,amsmath,amssymb,amsthm}  				    
\usepackage{fullpage,lineno}

\usepackage{color}

\newtheorem{theorem}{Theorem}[section]

\newtheorem{conjecture}[theorem]{Conjecture}

\newtheorem{problem}[theorem]{Problem}
\newtheorem{lemma}[theorem]{Lemma}

\newtheorem{corollary}[theorem]{Corollary}

\theoremstyle{definition}
\newtheorem{definition}[theorem]{Definition}

\def\VEC#1#2#3{#1_{#2},\ldots,#1_{#3}}

\def\SE#1#2#3{\sum_{#1=#2}^{#3}}

\def\FR#1#2{\frac{#1}{#2}}
\def\FL#1{\left\lfloor{#1}\right\rfloor}

\def\NN{{\mathbb N}}

\def\C#1{\left | #1 \right |}    

\def\cD{{\mathcal D}}
\def\cF{{\mathcal F}}
\def\hk{{\hat k}}

\begin{document}

\title{Acyclic graphs with at least $2\ell+1$ vertices
are $\ell$-recognizable}

\author{
Alexandr V. Kostochka\thanks{University of Illinois at Urbana--Champaign,
Urbana IL, and Sobolev Institute of Mathematics, Novosibirsk,
Russia: \texttt{kostochk@math.uiuc.edu}.  Research supported in part by NSF
grant DMS-1600592 and grants 18-01-00353A and 19-01-00682  of the Russian
Foundation for Basic Research.}\,,
Mina Nahvi\thanks{University of Illinois at Urbana--Champaign,
Urbana IL: \texttt{mnahvi2@illinois.edu}.}\,,
Douglas B. West\thanks{Zhejiang Normal Univ., Jinhua, China
and Univ.\ of Illinois at Urbana--Champaign, Urbana IL:
\texttt{dwest@illinois.edu}.  Supported by National Natural Science Foundation
of China grants NSFC 11871439, 11971439, and U20A2068.}\,,
Dara Zirlin\thanks{University of Illinois at Urbana--Champaign, Urbana IL 61801:
\texttt{zirlin2@illinois.edu}.  Research supported in part by Arnold O.
Beckman Campus Research Board Award RB20003 of the University of Illinois at
Urbana-Champaign.}
}

\date{\today}
\maketitle

\baselineskip 16pt

\begin{abstract}
The {\it $(n-\ell)$-deck} of an $n$-vertex graph is the multiset of subgraphs
obtained from it by deleting $\ell$ vertices.  A family of $n$-vertex graphs is
{\it $\ell$-recognizable} if every graph having the same $(n-\ell)$-deck as a
graph in the family is also in the family.  We prove that the family of 
$n$-vertex graphs having no cycles is $\ell$-recognizable when $n\ge2\ell+1$
(except for $(n,\ell)=(5,2)$).  It is known that this fails when $n=2\ell$.
\end{abstract}

\section{Introduction}

The {\it $k$-deck} of a graph is the multiset of $k$-vertex induced subgraphs.
We write this as the $(n-\ell)$-deck when the graph has $n$ vertices and the
focus is on deleting $\ell$ vertices.  An $n$-vertex graph is
{\it $\ell$-reconstructible} if it is determined by its $(n-\ell)$-deck.
It is an elementary observation, via a counting argument, that the $k$-deck of
a graph always determines its $(k-1)$-deck.  Therefore, an enhancement of the
Reconstruction Problem is to find for each graph the maximum $\ell$ such that 
it is $\ell$-reconstructible.  Kelly~\cite{Kel2} extended the classical
Reconstruction Conjecture of Kelly~\cite{Kel1} and Ulam~\cite{U} as follows:

\begin{conjecture}[\rm\cite{Kel2}]
For $l\in\NN$, there exists a threshold $M_\ell$ such that every graph with
at least $M_\ell$ vertices is $\ell$-reconstructible.
\end{conjecture}

Many reconstruction arguments have two parts.  First, one proves that the deck
determines that the graph is in a particular class or has a particular property.
When the $(n-\ell)$-deck determines this, the property is
{\it $\ell$-recognizable}.  Separately, using the knowledge that every
reconstruction from the deck has that property, one determines that only one
such graph has that deck.  That is, one proves that the family is {\it weakly
$\ell$-reconstructible}, meaning that no two graphs in the family have the same
deck.  Bondy and Hemminger~\cite{BH} introduced this separation into two steps
for the case $\ell=1$.

Here, toward $\ell$-reconstructibility of trees, we consider
$\ell$-recognizability of acyclic graphs.  We prove the following theorem.

\begin{theorem}
For $n\ge2\ell+1$ (except $(n,\ell)=(5,2)$), the $(n-\ell)$-deck of an
$n$-vertex graph determines whether the graph contains a cycle.
\end{theorem}

Our proof is constructive; the information we need in order to confirm that
every graph having $(n-\ell)$-deck $\cD$ is acyclic is computed from the deck.

Since the $(n-\ell)$-deck determines the $2$-deck, in this setting we also
know the number of edges.  This yields the following corollary.

\begin{corollary}
For $n\ge2\ell+1$ (except $(n,\ell)=(5,2)$), the $(n-\ell)$-deck of an
$n$-vertex graph determines whether the graph is a tree.
\end{corollary}

Spinoza and West~\cite{SW} determined for every graph $G$ with maximum degree
at most $2$ the maximum $\ell$ such that $G$ is $\ell$-reconstructible.  Their
full result is quite complicated to state, but a special case is that
for $n\ge2\ell+1$ (except $(n,\ell)=(5,2)$), every $n$-vertex graph with
maximum degree at most $2$ is $\ell$-reconstructible.  A path with $2\ell$
vertices has the same $\ell$-deck as the disjoint union of an $(\ell+1)$-cycle
and a path with $\ell-1$ vertices, as shown in~\cite{SW}, so the result
of~\cite{SW} and the result in the present paper are both sharp.

N\'ydl~\cite{N90} conjectured that trees with at least $2\ell+1$ vertices are
weakly $\ell$-reconstructible.  This conjecture would be sharp, since N\'ydl
presented two trees with $2\ell$ vertices having the same $\ell$-deck.  The two
trees are obtained from a path with $2\ell-1$ vertices by adding one leaf,
either to the central vertex of the path or to one of its neighbors.  Kostochka
and West~\cite{KW} used the results of~\cite{SW} to give a short proof that
these two trees have the same $\ell$-deck.  With our result, N\'ydl's
conjecture can be strengthened as follows.

\begin{conjecture}
For $\ell\ne2$, every tree with at least $2\ell+1$ vertices is
$\ell$-reconstructible.
\end{conjecture}

When $\ell=2$, the correct threshold is $6$ rather than $5$, since the union of
a $4$-cycle and an isolated vertex has the same $3$-deck as the tree obtained
by subdividing one edge of the $4$-vertex star.  Giles~\cite{Gil} proved that
trees with at least six vertices are $2$-reconstructible.  For $\ell=1$, the
only non-acyclic $n$-vertex graph having no cycle with length at most $n-1$ is
the $n$-vertex cycle, distinguished by its number of edges, so $n\ge3$ suffices.
In~\cite{KNWZ}, the present authors proved that $n\ge25$ suffices when $\ell=3$.
For $\ell\ge4$, it is not yet known whether there is a threshold $T_\ell$ such
that $n$-vertex trees are $\ell$-reconstructible when $n\ge T_\ell$.

Besides acyclicity, another fundamental property of trees is connectedness.
Spinoza and West~\cite{SW} proved that connectedness is $\ell$-recognizable
for $n$-vertex graphs when $n>2\ell^{(\ell+1)^2}$.  This threshold is surely
too large.  Manvel~\cite{Man} proved that connectedness is $2$-recognizable
for graphs with at least six vertices, and the present authors~\cite{KNWZ2}
proved that connectedness is $3$-recognizable for graphs with at least seven
vertices.  Spinoza and West~\cite{SW} suggested that (except for
$(n,\ell)=(5,2)$), connectedness is recognizable for $n$-vertex graphs when
$n\ge2\ell+1$.

As a first step toward the threshold on $n$ for $\ell$-recognizability of 
connectedness, one can consider $n$-vertex graphs whose $(n-\ell)$-deck has
only acyclic cards.  With $n-\ell\ge2$, we know the number of edges in any
reconstruction.  Our result in this paper settles the question for graphs with
$n-1$ edges, where connectedness and acyclicity are equivalent.
This motivates more detailed questions.

\begin{problem}
For $c,\ell\in\NN$, determine the smallest thresholds $N_{\ell,c}$ and
$N'_{\ell,c}$ such that for all $n$-vertex graphs with $n+c$ edges whose
$(n-\ell)$-vertex induced subgraphs are all acyclic,

(a) if $n\ge N_{\ell,c}$, then the $(n-\ell)$-deck determines whether the graph
is connected, and

(b) if $n\ge N'_{\ell,c}$, then the graph is connected.

\noindent
The thresholds when the cards are not required to be acyclic are also unknown.
\end{problem}

When $c=1$, the fact that a graph with $p$ vertices and at least $p+2$ edges
has girth at most $\FL{(p+2)/2}$ (see Exercise 5.4.36 of~\cite{West}, for
example) can be used to prove $N'_{\ell,1}\le 2\ell$.  That is, when the
graph has $n+1$ edges and the cards in the $(n-\ell)$-deck are acyclic, every
reconstruction is connected.  The threshold of $2\ell$ vertices is sharp when
$\ell$ is even, by the graph with $2\ell-1$ vertices consisting of an isolated
vertex plus four paths of length $\ell/2$ with common endpoints.  However, it
is possible that the threshold $N_{\ell,1}$ for determining whether all
reconstructions are connected may be smaller.

For $c=0$, we believe $N_{\ell,0}=2\ell-1$; Zirlin~\cite{Z} has proved this
for sufficiently large $\ell$.  The threshold is at least this much because
$C_{2\ell-2}$ and the disjoint union $C_{\ell-1}+C_{\ell-1}$ have the same
deck.  Zirlin~\cite{Z} also proved $N_{\ell,0}\le 2\ell+1$ for $\ell\ge3$.

\section{General Tools}

The proof for the special case $n=2\ell+1$ requires additional care beyond
the general argument.  In this section we develop tools useful for all cases.

Let $\cD$ be the $(n-\ell)$-deck of an $n$-vertex graph $G$ (we henceforth just
call it the ``deck'').  We will assume $n>2\ell$.  The members of $\cD$ are the
``cards'' in the deck.  We begin with a notion generalizing the degree list.

\begin{definition}
Given a vertex $v$ in a graph $G$, the {\it $k$-ball} at $v$, written $U_k(v)$,
is the subgraph induced by all vertices within distance $k$ of $v$ in $G$.  
A {\it $k$-vine} is a tree with diameter $2k$.  A {\it $k$-center} in a
graph $G$ is a vertex $v$ that is the center of a $k$-vine.
\end{definition}

The term ``$k$-vine'' continues the botanical theme of terminology about trees;
a vine has a main path from which the rest grows.
When $G$ is a forest, the maximal $k$-vine at a $k$-center $v$ is the
$k$-ball at $v$.  If the $k$-ball at $v$ does not contain a path of length
$2k$, then $v$ is not a $k$-center.  Our general approach is to consider an
acyclic and a non-acyclic graph having the same $(n-\ell)$-deck, show that
they have the same number of $k$-centers for an appropriate $k$, and obtain a
contradiction by showing that they cannot have the same number of $k$-centers.

In order to count $k$-centers using the $(n-\ell)$-deck, we will count the
$k$-centers whose $k$-balls have each size.  The special case $k=1$ will yield
the vertex degrees.  The key point is uniqueness of the maximal $k$-vine
containing a particular $k$-vine.

\begin{lemma}\label{buttermax}
In a graph $G$ with girth at least $2k+2$, every induced $k$-vine is contained
in a unique maximal $k$-vine.
\end{lemma}
\begin{proof}
Since $G$ has girth at least $2k+2$, every $k$-vine in $G$ is an induced
subgraph.  Since a $k$-vine $B$ contains a path $P$ of length $2k$, its
center $v$ is uniquely determined.  No $k$-vine with a center $w$ other
than $v$ contains $P$, because the distance from $w$ to one of the ends of
$P$ would exceed $k$.  Hence no $k$-vine with center $w$ contains $B$.
Thus the maximal $k$-vine containing $B$ can only be $U_k(v)$.
\end{proof}

Lemma~\ref{buttermax} will enable us to apply a general counting argument
that has been used in other contexts.  It combines ideas of Manvel~\cite{Man}
for $\ell$-reconstructibility and of Greenwell and Hemminger~\cite{GH} for
$1$-reconstructibility.  We follow the approach in a proof by Bondy and
Hemminger~\cite{BH} of the Greenwell--Hemminger result.

\begin{definition}
When $\cF$ is a class of graphs, an {\it $\cF$-subgraph} of a graph $G$ is an
induced subgraph of $G$ in $\cF$.  Let $s(F,G)$ denote the number of induced
subgraphs of $G$ isomorphic to $F$.  Let $m(F,G)$ denote the number of
occurrences of $F$ as a maximal $\cF$-subgraph of $G$.
An {\it absorbing family} for a graph $G$ is a family $\cF$ of graphs such that
every induced subgraph of $G$ belonging to $\cF$ lies in a unique maximal
induced subgraph of $G$ belonging to $\cF$.
\end{definition}

\begin{lemma}\label{counting}
Let $\cF$ be an absorbing family for an $n$-vertex graph $G$.  If $m(F,G)$ is
known for each $F\in\cF$ that has at least $n-\ell$ vertices, and the
$(n-\ell)$-deck of $G$ is known, then $m(F,G)$ can be determined from the
deck for all $F\in\cF$.
\end{lemma}
\begin{proof}
For an $\cF$-subgraph $F$ of $G$, let the {\it depth} of $F$ be the maximum
length $k$ of a chain $\VEC F0k$ of $\cF$-subgraphs such that $F=F_0$
and each is an induced subgraph of the next.

By hypothesis, $m(F,G)$ is known for all $F\in \cF$ with at least $n-\ell$
vertices.  From the $(n-\ell)$-deck, we also know the $j$-deck of $G$ for
$j<n-\ell$.  Hence we know all the $\cF$-subgraphs of $G$.  From this we
can determine the chains of $\cF$-subgraphs in $G$, so we can compute the depth
of $F$; let it be $k$.  We prove the claim by induction on $k$.

Let $j=\C{F}$, and suppose $j<n-\ell$.  If $k=0$, then $m(F,G)$ is the number
of copies of $F$ in the $j$-deck of $G$.  For $k>0$, group the induced copies
of $F$ in $G$ by which $\cF$-subgraph $H$ of $G$ is the unique maximal
$\cF$-subgraph containing this copy of $F$.  Now
$$
s(F,G)=\sum_{H\in\cF}s(F,H)m(H,G).
$$

When $s(F,H)\ne0$ and $F\ne H$, every chain of $\cF$-subgraphs starting at
$H$ can be augmented by adding $F$ at the beginning, so $H$ has smaller depth
than $F$.  By the induction hypothesis, we know every quantity in the displayed
equation other than $m(F,G)$.  The computation for $m(F,G)$ is the same for
every graph having the same $(n-\ell)$-deck as $G$.
\end{proof}

For example, the family of connected graphs is absorbing for every graph.
We formalize this application here because it was stated incorrectly in the
paper by Kostochka and West~\cite{KW} and because it illustrates the technique
we use for $k$-vines.  The special case for $\ell=1$ was observed by
Kelly~\cite{Kel2} using different methods.

\begin{corollary}
If $n>2\ell$, then $n$-vertex graphs having no component with more than
$n-\ell$ vertices are $\ell$-reconstructible, and this threshold on $n$ is 
sharp.  All $n$-vertex graphs having no component with at least $n-\ell$
vertices are $\ell$-reconstructible, with no restriction on $n$.
\end{corollary}
\begin{proof}
A graph with more than $2\ell$ vertices can only have one component with
at least $n-\ell$ vertices, and it has no component with more vertices if and
only if it has at most one connected $(n-\ell)$-card.  Hence the condition is
$\ell$-recognizable, and if there is a component with $n-\ell$ vertices it is
seen as a card.  Since the family of connected graphs is absorbing, by
Lemma~\ref{counting} graphs satisfying the condition are $\ell$-reconstructible.

The result is sharp, since $P_{\ell}+P_{\ell}$ and $P_{\ell+1}+P_{\ell-1}$ have
the same $\ell$-deck.  This follows from the result of Spinoza and
West~\cite{SW} that any two graphs with the same number of vertices and edges
whose components are all cycles with at least $k+1$ vertices or paths with
at least $k-1$ vertices have the same $k$-deck.
\end{proof}

By Lemma~\ref{buttermax}, the family of $k$-vines is absorbing for every graph
with girth at least $2k+2$ (the minimal $k$-vines in a graph all have $2k+1$
vertices).  This sometimes allows us to reconstruct the number of $k$-centers
from the deck.

\begin{corollary}\label{kcent}
Let $\cD$ be the $(n-\ell)$-deck of an $n$-vertex graph.  If every card in $\cD$
is acyclic, and every card in $\cD$ has radius greater than $k$, then all
reconstructions from $\cD$ have the same number of $k$-centers, which can
be computed from $\cD$.
\end{corollary}
\begin{proof}
A connected acyclic card with radius greater than $k$ has at least $2k+2$
vertices.  Hence $n-\ell\ge 2k+2$.  Since all cards are acyclic, every
reconstruction has girth at least $2k+3$.

Since every $(n-\ell)$-card has radius greater than $k$, every $k$-vine
has fewer than $n-\ell$ vertices.  Hence by examining the deck we see all
the $k$-vines and determine the maximum number of vertices in a
$k$-vine; call it $m$.  We have $m<n-\ell$.

Thus any reconstruction has no $k$-vines with $i$ vertices whenever
$i\ge n-\ell$.  By Lemma~\ref{buttermax}, Lemma~\ref{counting} implies that
the $(n-\ell)$-deck determines the numbers of maximal $k$-vines with
$i$ vertices for all $i$.  Finally, since the $k$-centers correspond
bijectively to maximal $k$-vines, the number of $k$-centers in any
reconstruction is determined by the deck.
\end{proof}

For the case $n=2\ell+1$ of our result, we will need the analogue of 
Corollary~\ref{kcent} for vertex degrees.  Manvel~\cite{Man} proved this
without the restriction to triangle-free graphs, but without the triangles
we obtain the result as a simple application of Lemma~\ref{counting}.

\begin{corollary}[\rm \cite{Man}]\label{deg1}
If the number of vertices of degree $i$ in an $n$-vertex triangle-free graph
$G$ is known whenever $i\ge n-\ell$, then the degree list of $G$ is
$\ell$-reconstructible.
\end{corollary}
\begin{proof}
The $1$-vines in $G$ are the stars with at least two edges.  With girth
at least $4$, Lemma~\ref{buttermax} allows use of Lemma~\ref{counting} to count
the vertices with each degree at least $2$.  Since we also know the $2$-deck
and $1$-deck, we know the numbers of edges and vertices, which also gives the
numbers of vertices of degrees $1$ and $0$.
\end{proof}

We will also need concepts for edges that are analogous to $k$-vines and
$k$-centers.

\begin{definition}
Given an edge $e$ in a graph $G$, the {\it $k$-eball} at edge $e$ is the
subgraph induced by all vertices within distance $k$ of either endpoint of
$e$ in $G$.  A {\it $k$-evine} is a tree with diameter $2k+1$.
A {\it $k$-central edge} in a graph $G$ is an edge $e$ whose $k$-eball
contains a $k$-evine whose center is the vertex set of $e$.
\end{definition}

Note that a $k$-evine has radius $k+1$.  Also, when the minimum radius
among cards is $k+1$, every card has diameter at least $2k+1$.

\begin{lemma}\label{ebuttermax}
In a graph $G$ with girth at least $2k+3$, every $k$-evine is contained
in a unique maximal $k$-evine.
\end{lemma}
\begin{proof}
A $k$-evine $B$ in $G$ contains a path $P$ with $2k+2$ vertices.  Since $G$ has
girth at least $2k+3$, $B$ is an induced subgraph of $G$.  The path $P$
determines a unique $k$-central edge $e$ in $B$, and no $k$-evine with
a different central edge can contain $P$.  Hence the unique maximal
$k$-evine containing $B$ is the $k$-eball for $e$.
\end{proof}

\begin{lemma}\label{kecent}
Let $n,k,\ell$ be positive integers with $2k+2\le n-\ell$.
If all cards in an $(n-\ell)$-deck $\cD$ are acyclic, with radius greater than
$k$, and no card has diameter $2k+1$, then the deck determines the
number of $k$-central edges in any $n$-vertex reconstruction, and this number
can be computed from the deck.
\end{lemma}
\begin{proof}
The acyclic cards with diameter $2k+1$ are the $k$-evines with $n-\ell$
vertices.  With such cards forbidden, no $k$-evine has more than $n-\ell$
vertices, because with $2k+2\le n-\ell$ we could delete leaves outside a
longest path to obtain a $k$-evine with $n-\ell$ vertices.

Since $n-\ell\ge 2k+2$, any reconstruction from $\cD$ has girth at 
least $2k+3$.  By Lemma~\ref{ebuttermax}, the family of $k$-evines 
is an absorbing family for any reconstruction.  Since no $k$-evines have
at least $n-\ell$ vertices, Lemma~\ref{counting} applies, and the 
deck determines the numbers of maximal $k$-evines with each number of vertices.
In particular, it determines the total number of maximal $k$-evines, and
this is the same as the number of $k$-central edges.
\end{proof}

\section{The Proof for $n\ge2\ell+2$}

We begin by developing a tool for bounding the number of $k$-centers in a 
forest in terms of its deck.  This tool does not depend on the relationship
between $n$ and $\ell$.

\begin{definition}\label{markdef}
Let $\cD$ be the $(n-\ell)$-deck $\cD$ of an $n$-vertex graph having a 
component with at least $n-\ell$ vertices.  Let $\hk$ be the minimum radius
among cards in $\cD$, and let $k=\hat k-1$.  {\it This henceforth fixes $k$.}

A {\it short card} is a card with radius $\hk$.  When $C$ is a short card with
center $z$ in the $(n-\ell)$-deck of an $n$-vertex forest $F$, the
{\it marking argument} describes a relationship between $k$-central vertices of
$F$ and vertices of $F$ outside $C$.  Each $k$-center $x$ other than $z$ that
is in the component of $F$ containing $C$ marks a vertex $x'$ at distance $k$
from $x$ along a path that extends the $z,x$-path in $F$.  If $x$ is not
adjacent to $z$, then $x'$ is outside $C$.  Furthermore, since $F$ is a forest,
every vertex outside $C$ is marked by at most one $k$-center.

For a short card $C$ with a center $z$, let $d_C$ denote the maximum number of
edge-disjoint paths of length $\hk$ in $C$ with common endpoint $z$.  Note
that short cards have diameter $2\hk$ (with unique center) or $2\hk-1$.
In the latter case, $d_C=1$ when viewed from either center.
\end{definition}

\begin{lemma}\label{marking}
If $C$ is a short card with center $z$ in the deck of a forest $F$, then the
number of $k$-centers in $F$ is at most $1+d_C+\ell$.  If equality holds, then
in the marking argument each vertex $x'$ of $F$ outside $C$ is marked by a
$k$-center $x$ not adjacent to $z$, and $F$ is a tree.
\end{lemma}
\begin{proof}
Let $F'$ be the component of $F$ containing $C$, and let $\ell'$ be the number
of vertices of $F'$ outside $C$.  The neighbors of $z$ in $C$ along paths of
length $\hk$ in $C$ are $k$-centers, as is $z$.  By the marking argument, $F'$
contains at most $\ell'$ additional centers.  In any component of $F$ other
than $F'$, the number of $k$-centers is strictly less than the number of
vertices, since vertices with degree at most $1$ cannot be $k$-centers.  The
total number of vertices of $F$ outside $C$ is $\ell$.  Summing over all the
components of $F$, the number of $k$-centers is at most $1+d_C+\ell$, with
equality only when $F$ is a tree.
\end{proof}

Our task is to study when a deck determines whether all reconstructions
have cycles.

\begin{definition}
We say that a deck $\cD$ is {\it ambiguous} if it is the $(n-\ell)$-deck of
both an $n$-vertex acyclic graph and an $n$-vertex non-acyclic graph.  Given
an ambiguous deck, we typically let $F$ and $H$ be $n$-vertex acyclic and
non-acyclic graphs having $(n-\ell)$-deck $\cD$, respectively.  All cards of an
amibiguous deck are acyclic, being induced subgraphs of a forest.  Hence when
$\cD$ is ambiguous the graph $H$ has girth at least $n-\ell+1$, and thus $\cD$
has connected cards (in fact, paths).
\end{definition}

Since all cards in $\cD$ are acyclic, and all cards have radius greater than
$k$, it follows from Corollary~\ref{kcent} that the number of $k$-centers is
the same for all reconstructions from $\cD$.  We will study the number of
$k$-central edges to eliminate the possibility of ambiguous decks
when $n\ge2\ell+2$.

First we exclude the case $\hk=1$, after which $k$ will be positive.

\begin{lemma}\label{hk>1}
If $\cD$ is ambiguous and $n\ge2\ell+2$, then $\hk>1$.
\end{lemma}
\begin{proof}
A card with radius $1$ is a star with $n-\ell$ vertices.
A non-acyclic reconstruction $H$ has a cycle with at least $n-\ell+1$ vertices.
Since $2n-2\ell+1\ge n+2$, a star with $n-\ell$ vertices must lie in the same
component of $H$ with any cycle.  Since $\cD$ also gives us the $2$-deck
of a forest, any reconstruction has at most $n-1$ edges.  Since $H$ contains
a cycle, $H$ must therefore be disconnected.

Since $H$ has girth at least $n-\ell+1$, which is at least $4$,
a star shares at most three vertices with a cycle.  Hence the number of 
vertices in their common component of $H$ is at least $(n-\ell+1)+(n-\ell)-3$.
Since this is at most $n-1$, we conclude $n\le 2\ell+1$.
\end{proof}

\begin{lemma}\label{k-eball}
If $\cD$ is ambiguous and $n\ge2\ell+2$, then $\cD$ has no card
with diameter $2k+1$.
\end{lemma}
\begin{proof}
Let $C$ be a card with diameter $2k+1$.  As remarked earlier, $d_C=1$.
By Lemma~\ref{marking}, at most $2+\ell$ vertices of $F$ are $k$-centers.
We have noted that all reconstructions from $\cD$ have the same number of
$k$-centers.  All vertices on a cycle in a reconstruction $H$ are $k$-centers in
$H$, and the cycle has length at least $n-\ell+1$.  Hence $2+\ell\ge n-\ell+1$,
which yields $n\le 2\ell+1$.
\end{proof}

\begin{theorem}\label{2l+2}
For $n\ge2\ell+2$, the family of $n$-vertex acyclic graphs is
$\ell$-recognizable.
\end{theorem}
\begin{proof}
Suppose that there is an ambiguous deck $\cD$ with reconstructions $F$ and $H$
as we have been discussing.  By Lemma~\ref{k-eball}, no card has diameter
$2k+1$.  Hence by Lemma~\ref{kecent} the number $s'$ of $k$-central edges is
the same in $F$ and $H$.  Let $C$ be a short card, and let $d=d_C$.

Note that $C$ has $d$ $k$-central edges incident to its unique center $z$.
An edge of $F$ in the component containing $z$ is a $k$-central edge if and
only if its endpoint farther from $z$ is a $k$-center.  In other components,
the number of $k$-central edges is less than the number of $k$-centers.
Hence Lemma~\ref{marking} implies $s'\le d+\ell$.

Since $C$ is a card in $\cD$, every reconstruction from $\cD$ has $d$
$k$-central edges with a common endpoint.  In $H$, only two of these can lie
on a particular cycle.  We have observed that $H$ has girth at least $2k+3$,
and hence every edge on a cycle in $H$ is a $k$-central edge.
Lemma~\ref{kecent} now yields $s'\ge n-\ell+1+d-2$.
The upper and lower bounds on $s'$ now require $n\le 2\ell+1$.
\end{proof}

The argument for $n\ge2\ell+2$ in this section is valid for all $\ell$.
In the next section we must restrict to $\ell\ge3$.

\section{The Case $n=2\ell+1$}

We begin with a result of independent interest about special trees.

\begin{definition}
A {\it branch vertex} in a tree is a vertex with degree at least $3$.
A {\it spider} is a tree having at most one branch vertex.
A {\it leg} of a tree is a path in the tree whose endpoints are a leaf and
a branch vertex.  When $d\ge3$, a spider whose branch vertex has degree $d$
is the union of $d$ paths with a common endpoint.  When those paths have
lengths $\VEC m1d$, we denote the spider by $S_{\VEC m1d}$; note that
$S_{\VEC m1d}$ has $1+\SE i1d m_i$ vertices.
\end{definition}

\begin{lemma}\label{spider}
When $n\ge2\ell+1\ge3$, an $n$-vertex spider contains at most $\ell+3$ paths
having exactly $n-\ell$ vertices, except for $S_{1,1,1,1}$ when $\ell=2$.
\end{lemma}
\begin{proof}
Let $G$ be an $n$-vertex spider with maximum degree $d$.  Let a {\it long path}
be a path with $n-\ell$ vertices.  We use induction on $\ell$.  When $\ell=1$,
an $n$-vertex tree has at most two paths with $n-1$ vertices, except that when
$n=4$ the spider $S_{1,1,1}$ has three such paths, still less than $\ell+3$.

For $\ell\ge2$, consider first the case that some leaf $x$ lies in at most one
long path.  Let $G'=G-x$.  Let $n'=|V(G')|=n-1$ and $\ell'=\ell-1$.  Since
$n\ge2\ell+1$, we have $n'>2\ell'+1$.  In particular, $(n',\ell')\neq (5,2)$,
so we can apply the induction hypothesis without considering the exception.
Thus $G'$ contains at most $\ell'+3$ paths with $n'-\ell'$ vertices.  That is,
$G$ has at most $\ell+2$ paths with $n-\ell$ vertices avoiding $x$.  Adding (at
most) one long path containing $x$ yields the desired bound for $G$.

In the remaining case, every leaf appears in at least two long paths.  Here we
argue directly, without needing the induction hypothesis.  Let $a$ be the
length of a shortest leg of $G$, and let $x$ be the leaf in a leg of length
$a$.  Since $x$ lies in two path-cards, $G$ must have at least two other legs
of length at least $n-\ell-1-a$, so $d\ge3$.

If $d\ge4$, then some fourth leg (with leaf $y$) also has length at least $a$.
Summing the lengths of these four legs yields $2n-2\ell-2\le n-1$, or
$n\le 2\ell+1$.  Since we consider only $n\ge2\ell+1$, equality holds,
requiring $G=S_{a,a,\ell-a,\ell-a}$ and $n-\ell=\ell+1$.  If $a<\ell-a$, then
exactly four long paths use $x$ or $y$ and a leg of length $\ell-a$, and
$\ell-2a+1$ long paths use the two long legs.  The total is $\ell-2a+5$, which
is at most $\ell+3$ since $a\ge1$.  If $a=\ell-a$, then there is also one long
path from $x$ to $y$, but now exceeding $\ell+3$ requires $a=1$ and $\ell=2$,
which occurs precisely for the exceptional case $S_{1,1,1,1}$.

Hence we may assume $d=3$.  The graph is $S_{a,b,c}$.  To have each leaf in two
path-cards, the lengths of any two legs must sum to at least $n-\ell-1$.  The
number of vertices in the union of two legs is the sum of their lengths plus
$1$, and the last $n-\ell-1$ vertices cannot start a long path.  Hence the
number of long paths is $2(a+b+c)+3-3(n-\ell-1)$, which equals $3\ell-n+4$.
Since $n\ge2\ell+1$, the number of long paths is at most $\ell+3$.
\end{proof}

For $\ell\ge4$, equality in Lemma~\ref{spider} requires $n=2\ell+1$ and occurs
for $S_{a,b,c}$ and for $S_{1,1,\ell-1,\ell-1}$.  When these are excluded, the
bound can be improved to $\ell+1$ except for the special case $S_{2,2,2,2}$
when $\ell=4$, but we will not need this stronger bound.

In the remainder of the paper we restrict to the setting $n=2\ell+1\ge7$.  We
maintain the notation and definitions for $F,H,\cD,\hk,k$ as in the previous
section.  In particular, an ambiguous deck $\cD$ is the $(n-\ell)$-deck of
both an acyclic $n$-vertex graph $F$ and a non-acyclic $n$-vertex graph $H$.
Also $\hk$ is the minimum radius among cards in $\cD$, and $k=\hk-1$.
We again begin by excluding the case $\hk=1$.

\begin{lemma}\label{nostar}
In an ambiguous deck $\cD$, no card is a star.  Hence $\hk>1$ and all
reconstructions have the same degree list.
\end{lemma}
\begin{proof}
Suppose that some card is a star, which is equivalent to $\hk=1$.
Since cards have $\ell+1$ vertices, there is no room for a star and a cycle
in separate components of a reconstruction $H$.  Since $H$ has girth at least
$n-\ell+1$, which is at least $5$, a star shares at most three vertices with a
cycle.  When they are in the same component $H'$ of $H$, the cycle must have
exactly $n-\ell+1$ vertices and share exactly three with the star, since
$n-\ell+1+\ell+1=n+2$ and $H$ is disconnected.  With $H'$ having $m$ vertices,
we have $n-\ell+1+\ell-2\le m\le n-1$, so equality holds, and we know $H$
exactly.

In particular, no other star has at least $n-\ell$ vertices.  Thus Manvel's
result (Corollary~\ref{deg1}) applies, and the deck determines the degree list
of every reconstruction.  However, the $2$-deck guarantees that every
reconstruction has the same number of edges.  We have found $H$ to be
unicyclic, with $n-1$ edges and an isolated vertex.  An acyclic reconstruction
$F$ is a tree, with no isolated vertex, so the degree lists are different.

This contradiction implies that no card is a star.  Hence $\hk>1$, and
therefore Corollary~\ref{deg1} again implies that the deck determines the
degree list.
\end{proof}

\begin{lemma}\label{maxdeg}
Graphs with an ambiguous deck $\cD$ have maximum degree at least $3$.
\end{lemma}
\begin{proof}
Since $n-\ell\ge4$, we can see in the deck whether there is a vertex of degree
at least $3$ in the reconstructions.  If not, then an acyclic reconstruction is
a subgraph of the path $P_{2\ell+1}$ and hence has at most $\ell+1$ cards that
are paths.  (We have $n-\ell=\ell+1$, and the last $\ell$ vertices cannot start
paths with $\ell+1$ vertices.)  On the other hand, a cycle with at least
$n-\ell+1$ vertices contains at least $\ell+2$ cards that are paths.
Hence for an ambiguous deck maximum degree at least $3$ is required.
\end{proof}

\begin{lemma}\label{kl-ineq}
When $\cD$ is an ambiguous deck, $2\hk\le\ell$.
\end{lemma}
\begin{proof}
By Corollary~\ref{kcent}, $\cD$ determines the number of $k$-centers in any
reconstruction.  When a reconstruction contains a cycle with length at least
$n-\ell+1$, every vertex on it is a $k$-center, so at most $\ell-1$ vertices
are not $k$-centers.  By Lemma~\ref{maxdeg}, an acyclic reconstruction has a
branch vertex, and we know that it also has a path with at least $n-\ell$
vertices.  Hence at least $2k+1$ vertices are not $k$-centers, consisting of
$k$ vertices at each end of a longest path plus one additional leaf.  We
conclude $2k+1\le \ell-1$, which yields $2\hk\le \ell$.
\end{proof}

Recall that by Corollary~\ref{kcent}, $\cD$ determines the number of
$k$-centers in $F$ and $H$.  As in Section 3, we will want to determine the
number of $k$-central edges and even the number of $\hk$-centers, but this
is more difficult when $n=2\ell+1$.  Again let $d_C$ be the maximum number of
edge-disjoint paths of length $\hk$ with common endpoint at a center $z$ of $C$
when $C$ is a short card.  The next two lemmas consider only one type of
reconstruction from $\cD$.

\begin{lemma}\label{ext}
Let $C$ be a short card in an $(n-\ell)$-deck $\cD$ having an acyclic 
reconstruction $F$.  If $F$ has at least $1+d_C+\ell$ $k$-centers, then 
for any two vertices $v_1,v_2 \in V(F)-V(C)$ at the same distance from
a center $z$ of $C$, the $z,v_1$-path and the $z,v_2$-path in $F$ share at most
one edge.  As a consequence, all vertices having distance at least $\hk$ from
$z$ have degree at most $2$.
\end{lemma}
\begin{proof}
By Lemma~\ref{marking}, the number of $k$-centers in $F$ is at most
$1+d_C+\ell$, with equality only when $F$ is a tree.  Let $Y$ be a set
of $d_C$ neighbors of $z$ along which edge-disjoint paths of length $\hk$
in $C$ depart from $z$.
By Lemma~\ref{marking}, the marking argument of Definition~\ref{markdef} marks
all vertices of $V(F)-V(C)$ using $k$-centers not in $Y\cup\{z\}$.

If the claim fails, then let $v_1$ and $v_2$ be two vertices of $F-V(C)$
closest to $z$ whose paths from $z$ share at least two edges.  For
$i\in\{1,2\}$, let $P_i$ be the
$z,v_i$-path in $F$, and let $x_i$ be the vertex at distance $k$ from 
$v_i$ along $P_i$.  Since $v_1$ and $v_2$ must be marked by distinct
$k$-centers outside $Y\cup\{z\}$, the vertices $x_1$ and $x_2$ are distinct.
Since $P_1$ and $P_2$ share at least two edges, $x_1$ and $x_2$ have distance
at least $3$ from $z$.  Now $z$ has distance at least $\hk+1$ to the neighbors
of $v_1$ and $v_2$ on the paths, so they are not in $C$, contradicting the
choice of $v_1$ and $v_2$.
\end{proof}

Let $d$ denote the maximum of $d_C$ over short cards $C$.

\begin{lemma}\label{Hcent}
When cards in $\cD$ are acyclic and there is a non-acyclic reconstruction $H$,
the number of $k$-centers and the number of $k$-central edges in $H$ are both
at least $n-\ell+d-1$, which is $d+\ell$ when $n=2\ell+1$.  If also $\cD$
is ambiguous, then $H$ is unicyclic.
\end{lemma}
\begin{proof}
Let $C$ be a short card with $d_C=d$.  Since $n-\ell>n/2$, the card $C$ lies in
the component of $H$ containing any cycle.  The induced subgraph $C$ has $d$
$k$-centers with a common neighbor and $d$ $k$-central edges with a common
endpoint (in any reconstruction).  At most two of these (in either case) lie on
a cycle in $H$, but with $2\hk\le n-\ell$, all the vertices on a cycle in $H$
are $k$-centers (and its edges are $k$-central edges), and there are at least
$n-\ell+1$ of each.  Hence at least $n-\ell+1+d-2$ vertices in $H$ are
$k$-centers, and at least the same number of edges are $k$-central edges.

If $H$ has more than one cycle, then since each has at least $\ell+2$ vertices
and $H$ has only $2\ell+1$ vertices, any two cycles share at least three
vertices.  When two cycles share at least two vertices, their union contains
two vertices joined by three edge-disjoint paths.  Let $R$ be the union of
three such paths with least total length.  Note that $R$ contains three cycles,
with each edge of $R$ appearing in two of the cycles.  Let the lengths of the
three paths be $a$, $b$, and $c$.  Summing the girth requirement over the three
cycles yields $2(a+b+c)\ge 3(n-\ell+1)$.

There are $a+b+c-1$ vertices in $R$, and every vertex on a cycle is a
$k$-center, since the girth is at least $\ell+2$.
Also, the short card $C$ provides $d$ $k$-centers with a common neighbor,
of which at least $d-3$ are not in $R$, since $R$ has maximum degree
$3$ and is chosen with minimum number of edges.  Hence $H$ has at least
$(a+b+c-1)+(d-3)$ $k$-centers.

By Corollary~\ref{kcent}, all reconstructions have the same number of 
$k$-centers.  By Lemma~\ref{marking}, any acyclic reconstruction has at
most $1+d+\ell$ $k$-centers.  Since $\cD$ is ambiguous, there is such a 
reconstruction.  With $s$ denoting the number of $k$-centers, we obtain
$$
\textstyle{\FR32}(n-\ell+1)+d-4\le a+b+c+d-4\le s\le 1+d+\ell.
$$
With $n-\ell+1=\ell+2$, this simplifies to $3(\ell+2)-8\le 2+2\ell$
and then $\ell\le 4$.

To eliminate the remaining cases of small $\ell$, note that $a+b+c\le n-1$,
since there is an acyclic reconstruction.  Hence
$2n-2\ge 2(a+b+c)\ge 3(n-\ell+1)$.  With $n=2\ell+1$, this simplifies
to $4\ell\ge 3\ell+6$, which requires $\ell\ge6$.

Hence there is no value of $\ell$ that allows $H$ to have more than once cycle.
\end{proof}

We can now prove the analogue of Lemma~\ref{k-eball} for the case $n=2\ell+1$.

\begin{lemma}\label{no2k+1}
If $\cD$ is ambiguous and $n=2\ell+1$, then $\cD$ has no card
with diameter $2k+1$.
\end{lemma}
\begin{proof}
Let $C$ be a connected card with diameter $2k+1$; note that $d_C=1$.
Since $d_C=1$, by Lemma~\ref{marking} $F$ has at most $\ell+2$ $k$-centers.
By Corollary~\ref{kcent}, $F$ and $H$ have the same number of $k$-centers, $s$.
Let $C'$ be a short card chosen to maximize $d_{C'}$; let $d=d_{C'}$.
By Lemma~\ref{Hcent}, $H$ has at least $d+\ell$ $k$-centers, but also at least
$\ell+2$ along its unique cycle $Q$.  Hence $\ell+2\ge s\ge \ell+\max\{2,d\}$.
If $d\ge3$, then we have a contradiction.

Hence we may assume $d\le2$ and $s=\ell+2$.  Therefore Lemma~\ref{ext}
applies to $C$.  First, by Lemma~\ref{marking}, in the marking argument every
vertex of $F$ outside $C$ is marked by a vertex outside the central edge $e$ of
$C$, and $F$ is a tree.  Vertices within distance $k$ of $e$ cannot be
marked, and therefore $C$ contains all such vertices.

Thus vertices of $F$ outside $C$ have distance at least $\hk$ from each
center of $C$.  Since we can apply Lemma~\ref{ext} with either vertex of the
central edge $e$ being the center $z$, the paths in $F$ from any two vertices
outside $C$ to $e$ cannot meet before reaching $e$.

{\bf Case 1:} {\it Two paths from vertices outside $C$ reach $e$ at the same
endpoint $z$.}
Let $v$ and $v'$ be the vertices at distance $\hk$ from $z$ on these two paths.
Since $C$ has diameter $2k+1$, a third path of length $\hk$ reaches $z$ through
$e$.  If $C$ has vertices $w$ and $w'$ outside these three paths, then deleting
$\{w,w'\}$ and adding $\{v,v'\}$ produces a short card exhibiting $d\ge3$, a
contradiction.  Hence $C$ is either the spider $S_{k,k,k+1}$ with three legs
and $3k+2$ vertices or is obtained from $S_{k,k,k+1}$ by adding one vertex.

By Lemma~\ref{ext}, $F$ has no branch vertices outside $C$, and 
no vertex of $C$ has two neighbors outside $C$.  Thus $F$ is obtained from $C$
only by extending paths from leaves of $C$ at distance $k$ from $e$.  Hence $F$
also is a spider or a spider plus one extra vertex.

Again call a path with $n-\ell$ vertices a {\it long path}.  If $F$ is a 
spider, then by Lemma~\ref{spider} $F$ has at most $\ell+3$ long paths,
since $n=2\ell+1$ with $(n,\ell)\ne(5,2)$.  If $F$ consists of a spider plus
one leaf $v$, then by Lemma~\ref{spider} $F-v$ has at most $\ell+2$ paths with
$n-\ell$ vertices, since $n-\ell = (n-1)-(\ell-1)$ and $n-1>2(\ell-1)+1$
with $(n-1,\ell-1)\ne(5,2)$.  Since $F-v$ is a spider with branch vertex of
degree $3$, adding $v$ adds at most three long paths, so $F$ has at most
$\ell+5$ long paths.

By Lemma~\ref{Hcent}, $H$ has a unique cycle $Q$.  With $n=2\ell+1$, there
are at least $\ell+2$ vertices in $Q$.  Since $H$ is disconnected, the
component $H'$ of $H$ containing $Q$ has at most $\ell-2$ vertices outside $Q$.
Hence every vertex of $H'$ has distance at most $\ell-2$ from $Q$.  From each
vertex of $H'$ outside $Q$, we can follow a shortest path to $Q$ and then turn 
either direction along $Q$ to complete a long path.  With also at least
$\ell+2$ long paths in $Q$, the number of long paths in $H$ is at least
$\ell+2+2t$, where $t$ is the number of vertices of $H'$ outside $Q$.

Since $H'$ is unicyclic, $C\cap Q$ is a single path along $Q$.  Since
$C$ has diameter $2k+1$, it shares at most $2k+2$ vertices with $Q$.
With $\C{V(C)}\in\{3k+2,3k+3\}$, we obtain $t\ge k$ if $F$ is a spider, and
$t\ge k+1$ if $F$ is a spider plus one vertex.  Thus $H$ has at least
$\ell+2+2k$ long paths in the former case and at least $\ell+4+2k$ in the
latter.  With $k\ge1$, in each case $H$ has more long paths than $F$, which is
a contradiction since the long paths are cards in $\cD$.

{\bf Case 2:} {\it Each endpoint of $e$ is reached by at most one path from
outside $C$.}
In this case, $F$ extends $C$ by at most two paths, one grown from each end of
$e$.  The path joining them in $C$ has $2k+2$ vertices.  Together, these paths
form a path $P$ with $\ell+2k+2$ vertices.  It contains only $2k+2$ long paths,
since the last $\ell$ vertices cannot start a path with $\ell+1$ vertices.  The
vertices of $F$ outside $P$ are $\ell+1-(2k+2)$ vertices of $C$.  Each such
vertex can start a long path only by traveling to $P$ and turning one way or
the other along $P$.  Hence the total number of long paths in $F$ is at most
$2(\ell+1)-(2k+2)$, which equals $2\ell-2k$.

In $H$, again $C$ lies in the component $H'$ containing $Q$.  Any vertex
$v\in V(H')-V(Q)$ is within distance $\ell-1$ from $Q$.  Hence from $v$ a long
path can be followed to $Q$ and then along $Q$ in either direction.  Also at
least $n-\ell+1$ long paths lie in $Q$.  As before, $C$ shares at most $2k+2$
of its $\ell+1$ vertices with $Q$, since $H$ is unicyclic.  Hence the total
number of long paths in $H$ is at least $(n-\ell+1)+2(\ell+1-2k-2)$, which
equals $3\ell-4k$.

Since $2k+2\le n-\ell=\ell+1$, we have $2k<\ell$, which is equivalent to
$2\ell-2k<3\ell-4k$.
\end{proof}

We may henceforth assume that all short cards have diameter $2k+2$, so they are
$\hk$-vines.

\begin{lemma}\label{hkvine}
When $\cD$ is ambiguous and $n=2\ell+1$, every reconstruction has
$d+\ell$ $k$-central edges, every acyclic reconstruction $F$ has at least
$1+d+\ell$ $k$-centers, and every $\hk$-vine in $F$ has at most $n-\ell$
vertices.
\end{lemma}
\begin{proof}
By Lemma~\ref{no2k+1}, no card has diameter $2k+1$.  Hence every $k$-evine has
at most $n-\ell-1$ vertices, and by Lemma~\ref{kecent} the number $s'$ of
$k$-central edges is the same in $F$ and $H$.  By Lemma~\ref{Hcent}, $H$ has at
least $d+\ell$ $k$-central edges.  Hence $F$ also has $d+\ell$ $k$-central
edges.

Since any short card has diameter $2k+2$, it has a unique central vertex,
and this vertex is a $k$-center.  In every $k$-central edge in the card, the
endpoint farther from the center of the card is also a $k$-center.  Hence $F$
has at least $1+d+\ell$ $k$-centers.

If $F$ contains a $\hk$-vine with more than $n-\ell$ vertices, then we can
delete vertices to obtain a $\hk$-vine $B$ with $n-\ell+1$ vertices (since
$2\hk\le n-\ell$).  A longest path in $B$ has $2k+3$ vertices.  Deleting a leaf
of $B$ yields a card of $F$.  Since no card has diameter $2k+1$, $B$ has a leaf
$v$ such that $B-v$ is a card $C$ that is a $\hk$-vine.  Let $z$ be the common
center of $B$ and $C$.

Under the marking argument, vertices at distance $\hk$ from $z$ are marked only
by $k$-centers adjacent to $z$.  Since all vertices of $C$ are within distance
$\hk$ of $z$, at most $d$ vertices of $C$ and the $\ell$ vertices of
$V(F)-V(C)$ can be marked.  Hence $F$ has at most $1+d_C+\ell$ $k$-centers.

We conclude that equality holds, so $d_C=d$ and every vertex outside $C$ is
marked.  Now we ask where is $v$?  Since $v\in V(B)$, it is within distance
$\hk$ of $z$; to be marked it must have distance exactly $\hk$ from $z$, 
marked by a neighbor of $z$.  Thus if $C$ has any leaf vertex $w$ that is
not on one of the $d$ edge-disjoint paths from $z$ or the $z,v$-path, then
replacing $w$ with $v$ in $C$ yields a card $C'$ with $d_{C'}>d$.  By the
choice of $d$, we conclude that $C$ is a spider with $d$ legs of length
$\hk$ and one leg of length $k$, and $v$ extends that leg to length $\hk$.

All vertices outside $B$ have distance greater than $\hk$ from $z$.
By Lemma~\ref{ext}, $F$ is a spider and has at most $\ell+3$ cards that are
paths.  The unique cycle $Q$ in $H$ contains at least $\ell+2$ cards that are
paths.  With one more vertex in the component $H'$ of $H$ containing $Q$,
we obtain $\ell+4$ such cards and a contradiction.  Since no short card has
diameter $2k+1$, we have $d\ge2$, and the card $C$ has at least two paths of
length $\hk$ and one path of length $k$ from the center, edge-disjoint.  Also
$C$ and $Q$ each contain more than half of $V(H)$ and must intersect.  Now the
vertex of degree at least $3$ in $C$ guarantees a vertex of $H'$ outside $Q$.
\end{proof}

\begin{lemma}\label{Hvine}
When $\cD$ is ambiguous and $n=2\ell+1\ge7$, in any non-cyclic reconstruction
$H$ no $\hk$-vine has more than $n-\ell$ vertices.
\end{lemma}
\begin{proof}
By Lemma~\ref{hkvine}, every reconstruction has $d+\ell$ $k$-central edges,
where $d$ is the maximum of $d_C$ over short cards $C$.  By Lemma~\ref{Hcent},
$H$ has a unique cycle, $Q$.

Let $H'$ be the component of $H$ containing $Q$.  Every card has $\ell+1$
vertices and hence intersects $Q$, so $H'$ contains a short card $C$.  Since
$H$ is unicyclic, $C\cap Q$ is connected.  Since $C$ has diameter $2k+2$, we
conclude that $C$ shares at most $2k+3$ vertices with $Q$.

Now $|V(H')|\ge (n-\ell+1)+(n-\ell)-(2k+3) = n-(2k+1)$.  Thus at most $2k+1$
vertices lie outside $H'$.  If some $k$-center lies outside $H'$, then there is
only one, and its component is a path with $2k+1$ vertices.

Each vertex on $Q$ is a $k$-center, and each edge on $Q$ is a $k$-central edge.
For every $k$-center $v$ in $H'-V(Q)$, the edge leaving $v$ on the path to $Q$
is a $k$-central edge (and the endpoint farther from $Q$ in any $k$-central
edge of $H'-V(Q)$ is a $k$-center).  Thus the difference between the numbers of
$k$-centers and $k$-central edges is the same in $H'$ as in $Q$, where it is
$0$.

We proved in Lemma~\ref{hkvine} that $F$ and $H$ have the same number of
$k$-central edges and that $F$ has more $k$-centers than $k$-central edges.
We also know from Corollary~\ref{kcent} that $F$ and $H$ have the same number
of $k$-centers.  Hence $H$ also has more $k$-centers than $k$-central edges.
This requires $H=H'+P_{2k+1}$, with $C$ and $Q$ sharing $2k+3$ vertices in $H'$.

Now consider in $H$ a $\hk$-vine $B$ with center $v$.  Since $B$ is connected
with at least $2\hk+1$ vertices, it must lie in $H'$ and omit the $2k+1$
vertices of the outside path.  Since $H'$ is unicyclic, $B\cap Q$ is connected.
Being a tree with diameter $2\hk$, the tree $B$ contains at most $2\hk+1$
vertices among the $\ell+2$ vertices of $Q$.  Therefore, $B$ omits at least 
$\ell+2-(2k+3)$ vertices of $H'$ and $2k+1$ vertices outside $H'$,
which means $B$ has at most $n-\ell$ vertices.
\end{proof}

\begin{theorem}
For $n\ge2\ell+1\ge7$, the family of $n$-vertex acyclic graphs is
$\ell$-recognizable.
\end{theorem}
\begin{proof}
By Theorem~\ref{2l+2}, we may assume $n=2\ell+1$.  By Lemmas~\ref{hkvine}
and~\ref{Hvine}, in every reconstruction no $\hk$-vine has more than $n-\ell$
vertices.  Thus all $\hk$-vines are seen in the deck.  The deck then provides
the number of maximal $\hk$-vines with $n-\ell$ vertices, and none are larger.
Also any reconstruction has girth at least $n-\ell+1$, which by
Lemma~\ref{kl-ineq} is at least $2\hk+2$.  Hence by Corollary~\ref{kcent}
$\cD$ determines the number of $\hk$-centers in any reconstruction.

Any short card $C$ has radius $\hk$ and diameter $2\hk$, so it is a $\hk$-vine
with a unique center $z$.  We now modify the marking argument of 
Definition~\ref{markdef} and Lemma~\ref{marking} so that in $F$ each
$\hk$-center $x$ other than $z$ marks a vertex at distance $\hk$ from $x$ along
an extension of the $z,x$-path in $F$.  Since $C$ has radius $\hk$, the marked
vertices are outside $C$, and a vertex can only be marked by one $\hk$-center.
Hence the number of $\hk$-centers in $F$ is at most $1+\ell$.

However, every vertex on a cycle in $H$ is a $\hk$-center, since
$2\hk\le n-\ell$, so the number of $\hk$-centers in $H$ is at least $n-\ell+1$,
which equals $\ell+2$.  This contradicts that $F$ and $H$ have the same number
of $\hk$-centers and completes the proof.
\end{proof}

\end{document}